\documentclass[draft]{siamltex1213}

\usepackage{amsfonts}
\usepackage{amssymb}
\usepackage{amsmath}
\usepackage{cite}
\usepackage{enumerate}

\usepackage[prependcaption,colorinlistoftodos]{todonotes}

\DeclareMathOperator{\im}{im}

\DeclareMathOperator{\inte}{int}
\DeclareMathOperator{\clo}{cl}

\DeclareMathOperator{\conv}{conv}
\DeclareMathOperator{\cone}{cone}
\DeclareMathOperator{\dom}{dom}

\let\leq\leqslant
\let\geq\geqslant
\let\emptyset\varnothing

\newcommand{\calJ}{\ensuremath{\mathcal{J}}}
\newcommand{\calK}{\ensuremath{\mathcal{K}}}

\newcommand{\calR}{\ensuremath{\mathcal{R}}}

\newcommand{\calT}{\ensuremath{\mathcal{T}}}

\newcommand{\calV}{\ensuremath{\mathcal{V}}}

\newcommand{\calY}{\ensuremath{\mathcal{Y}}}

\newcommand{\bmat}{\begin{matrix}}
\newcommand{\emat}{\end{matrix}}
\newcommand{\bbm}{\begin{bmatrix}}
\newcommand{\ebm}{\end{bmatrix}}
\newcommand{\bpm}{\begin{pmatrix}}
\newcommand{\epm}{\end{pmatrix}}
\newcommand{\bse}{\begin{subequations}}
\newcommand{\ese}{\end{subequations}}
\newcommand{\beq}{\begin{equation}}
\newcommand{\eeq}{\end{equation}}
\newcommand{\ben}{\renewcommand{\labelenumi}{\arabic{enumi}.}
\renewcommand{\theenumi}{\arabic{enumi}}\begin{enumerate}}
\newcommand{\een}{\end{enumerate}}
\newcommand{\beni}{\renewcommand{\labelenumi}{\roman{enumi}.}
\renewcommand{\theenumi}{\roman{enumi}}\begin{enumerate}}
\newcommand{\eeni}{\end{enumerate}\renewcommand{\labelenumi}{\arabic{enumi}.}
\renewcommand{\theenumi}{\arabic{enumi}}}
\newcommand{\bena}{\renewcommand{\labelenumi}{\alpha{enumi}.}
\renewcommand{\theenumi}{\alpha{enumi}}\begin{enumerate}}
\newcommand{\eena}{\end{enumerate}\renewcommand{\labelenumi}{\arabic{enumi}.}
\renewcommand{\theenumi}{\arabic{enumi}}}
\newcommand{\bit}{\begin{itemize}}
\newcommand{\eit}{\end{itemize}}
\newcommand{\bthe}{\begin{theorem}}
\newcommand{\ethe}{\end{theorem}}
\newcommand{\blem}{\begin{lemma}}
\newcommand{\elem}{\end{lemma}}
\newcommand{\bprop}{\begin{proposition}}
\newcommand{\eprop}{\end{proposition}}
\newcommand{\bex}{\begin{example}}
\newcommand{\eex}{\end{example}}
\newcommand{\bas}{\begin{assumption}}
\newcommand{\eas}{\end{assumption}}
\newcommand{\bre}{\begin{remark}}
\newcommand{\ere}{\end{remark}}
\newcommand{\bcor}{\begin{corollary}}
\newcommand{\ecor}{\end{corollary}}
\newcommand{\bdfn}{\begin{definition}}
\newcommand{\edfn}{\end{definition}}

\newcommand{\inv}{\ensuremath{^{-1}}}
\newcommand{\pset}[1]{\ensuremath{\{#1\}}}

\newcommand{\R}{\ensuremath{\mathbb R}}

\DeclareMathOperator{\gr}{gr}

\newtheorem{example}[theorem]{Example}
\newtheorem{asmp}[theorem]{Assumption}
\newtheorem{remark}[theorem]{Remark}
\newtheorem{cor}[theorem]{Corollary}
\newtheorem{prop}[theorem]{Proposition}

\title{A CHARACTERIZATION OF CONTROLLABILITY FOR DISCRETE-TIME LINEAR SYSTEMS WITH CONVEX CONSTRAINTS}

\author{M.~D. Kaba\footnotemark[3]
\and M.~K. Camlibel\footnotemark[1] \footnotemark[2]}

\begin{document}
\maketitle

\renewcommand{\thefootnote}{\fnsymbol{footnote}}

\footnotetext[1]{Johann Bernoulli Institute for Mathematics and Computer Science, University of Groningen, Nijenborgh 9, 9747 AG, Groningen, The Netherlands. (email: {\tt m.k.camlibel@rug.nl})}
\footnotetext[3]{Department of Applied Mathematics and Statistics, Johns Hopkins University, Baltimore, MD USA. (email: {\tt mkaba1@jhu.edu}). This work was done while the author was a Ph.D. student at University of Groningen.}

\renewcommand{\thefootnote}{\arabic{footnote}}

\section{Introduction}
In the study of actual physical systems the constraints arise naturally. Therefore, a reformulation of some fundamental concepts of control theory, like reachability, null-controllability, and controllability  in the presence of constraints, is a necessity. In order to improve the well-established theory of linear systems in this vein, throughout the years, many authors have considered various different problems in the field. Among all the valuable work, the work \cite{Ng:86} of Nguyen and \cite{So:84} of Sontag, which address the null-con\-trol\-la\-bi\-li\-ty problem; the works \cite{Br:72}, \cite{Ev:85}, \cite{EM:77}, and \cite{Sa:73}, which address the controllability problem; and the works \cite{SHS:02, SSS:04, SSSS:04, SSS:03, WSS:10} of Saberi et al., which address the stabilization problem, worth mentioning. The earlier works on the controllability consider the input constraints only. The controllability problem in the presence of state constraints, was not addressed until the appearance of \cite{HC:07}. In this paper, Heemels and Camlibel characterize the controllability of a constrained continuous-time linear system which is right-invertible. They further assume that the constraint set is a solid closed polyhedral cone.

For constrained linear systems, the reachability problem was initially treated within the controllability problem. However, it was discussed in a very limited setting, where only the input constraints were considered. In a more general setting, the reachability of strict closed convex processes were characterized by Aubin et al. in their remarkable paper \cite{AFO:86}. The discrete-time version of \cite{AFO:86} was presented by Phat and Dieu in \cite{PD:94}. In the light of these works, in \cite{KaCa:15a} we provided an almost complete spectral characterization of the controllability for discrete-time linear systems with mixed input and state constraints, where the constraint set was assumed to be a convex cone containing the origin. We divided the problem into three cases and gave a characterization of the reachability in two of them. When a certain subspace $\mathcal{K}(\Sigma)$, closely related to the right-invertibility of the discrete-time system, intersected the interior of the constraint cone, we showed that classical characterization of reachability in terms of invariance properties of dual constrained system can be extended. When $\mathcal{K}(\Sigma)$ and the constraint cone had a trivial intersection, we still managed to characterize the reachability. However, it was shown that this new case required a new sort of characterization. This led to a characterization of the controllability for these two cases. The remaining was a pathological case where the intersection of the subspace and the constraint set was nontrivial and contained in the boundary of the constraint set. In general, no characterization is known for this case, and in \cite{KaCa:15a}, we showed that the known characterizations cannot be extended to this case.

In this paper, we aim to improve our results in \cite{KaCa:15a}, by removing the conicity assumption on the constraint set. We show that a few of the theorems of \cite{KaCa:15a} generalize immediately. However, the characterization of reachability in terms of the invariance properties of the dual constrained system is more involved. Not only the conditions are increased in number, but also the hyperbolicity of the constrained set is crucial. Moreover, the necessity of these dual conditions can only be guaranteed when $\mathcal{K}(\Sigma)$ and the convex constraint set span the whole output space, which is not anymore an immediate consequence of the fact that $\mathcal{K}(\Sigma)$ and the interior of the constrained set had a nontrivial intersection.

We divide our presentation into seven parts. After this introduction we will formulate the problem we would like to discuss. Then we will review the preliminary concepts relevant to our discussion. In the section following this review, we will present our results regarding the reachability. Then we will provide a theorem showing the equivalence of reachability and controllability. After we present the proofs for our results, we will end with some concluding remarks.

\section{Problem formulation}\label{4PF}
Consider the discrete-time linear system
\begin{subequations}\label{e:4dts}
\begin{align}
 x_{k+1} &=Ax_k+Bu_k\label{e:4dts-1}\\
 y_k &= Cx_k+Du_k\label{e:4dts-2}
\end{align} 
where the input $u$, state $x$, and output $y$ have dimensions $m$, $n$, and $s$, respectively.  We denote this system by $\Sigma=\Sigma(A,B,C,D)$.

Given a convex set $\mathcal{Y}\subseteq \mathbb{R}^s$ containing the origin, consider the system \eqref{e:4dts-1}-\eqref{e:4dts-2} together with the output constraint 
\beq\label{e:4dts-3}
y_k\in \mathcal{Y}.
\eeq
\end{subequations}
We will denote the constrained system \eqref{e:4dts} by $(\Sigma,\mathcal{Y})$.

We say that a vector $x\in \mathbb{R}^n$ a \emph{feasible state\/} if there exist $\{u_k\}_{k\geq 0}$ and $\{x_k\}_{k\geq 0}$ with $x_0=x$ such that
\begin{align*}
 x_{k+1} &=Ax_k+Bu_k,\\
 \mathcal{Y} &\ni Cx_k+Du_k,
\end{align*} 
for all $k\geq 0$. A vector $x\in \mathbb{R}^n$ is called a \emph{reachable state\/} if there exist an integer $\ell$ with $\ell\geq 1$, $\{u_k\}_{0\leq k\leq \ell-1}$ and $\{x_k\}_{0\leq k\leq \ell}$ with $x_0=0$ and $x_\ell=x$ such that
\begin{align*}
 x_{k+1} &=Ax_k+Bu_k\\
 \mathcal{Y} &\ni Cx_k+Du_k
\end{align*} 
for all $k$ with $0\leq k\leq \ell-1$. 


If all the feasible states of a constrained system $(\Sigma,\mathcal{Y})$ are reachable, then we say that the system $(\Sigma,\mathcal{Y})$ is \emph{reachable}. 

In this paper, we investigate the conditions which are equivalent to the reachability of a given system $(\Sigma, \mathcal{Y})$.

\section{Preliminaries}
This section is devoted to review the notation and basic notions/results from convex analysis as well as geometric control theory.
\subsection{Convex sets}
We refer to the \emph{Preliminaries} section of \cite{KaCa:15a} for most of the background material. Here we will mainly cover the concepts which cannot be found there.

We will denote the closed unit ball by $\mathbb{B}$. It is the set of all vectors of (Euclidean) norm less than or equal to one in a given real vector space. If $C_1$ and $C_2$ are two subsets of $\mathbb{R}^n$, then $\conv(C_1\cup C_2)$ denotes the \emph{convex hull} of $C_1$ and $C_2$, i.e. the smallest convex set containing $C_1$ and $C_2$.

Let $\mathbb{R}_{\geq 0}$ denote the set of non-negative real numbers and let $C$ be a non-empty convex subset of $\mathbb{R}^h$. We denote the \emph{recession cone} of $C$ by $0^+C$.  That is
$$0^+C = \{y\in \mathbb{R}^h\mid x+\mu y\in C\text{ for all }x\in C \text{ and }\mu \in \mathbb{R}_{\geq 0}\}.$$
The convex set $C$ is bounded if and only if $0^+ C= \{0\}$. The \emph{interior} of $C$, denoted $\inte(C)$, on the other hand, is given by
$$\inte(C) = \{x\in C\mid \exists \epsilon > 0 \text{ s.t. } x+\epsilon\mathbb{B}\subseteq C\}.$$
If the interior of $C$ is non-empty, we say that $C$ is \emph{solid}. The \emph{conic hull} of $C$ will be denoted by $\cone(C)$. It is the smallest convex cone containing $C$ and the origin.

The operation $\langle\cdot,\cdot\rangle$ will denote the standard inner product. We define the \emph{polar set} of $C$ as
\begin{equation}
 C^\circ=\{q\in \mathbb{R}^h\mid\langle q,x\rangle\leq 1 \text{ for all }x \in C\}.
\nonumber\end{equation}
When $C$ is a convex cone containing the origin, the polar set of $C$ coincides with the \emph{negative polar cone} $C^-$ of $C$. We recall that
$$C^-=\{q\in \mathbb{R}^h\mid\langle q,x\rangle\leq 0 \text{ for all }x \in C\}$$
and the \emph{positive polar cone} $C^+$ is then the negative of $C^-$. Section $16$ of \cite{Roc:70} gives an excellent summary of the properties of the polar sets. We refer to this section for further details.

Another important dual concept is the \emph{barrier cone} of $C$, which is defined as
$$C^\mathrm{b} =\{q\in \mathbb{R}^h\mid \sup_{x\in C} \langle q,x\rangle < +\infty\}.$$
It is well-known that
$$\clo(C^\mathrm{b})= (0^+C)^-.$$

The convex set $C$ is called \emph{hyperbolic}, if there exists $\mu\in \mathbb{R}_{\geq 0}$ such that
$$C \subseteq \mu\mathbb{B} + 0^+C.$$ 
In this case, the barrier cone is always closed (\cite[Prop. 5, p.183]{Ba:83}), and we have
$$C^\circ \supseteq \mu^{-1}\mathbb{B}\cap C^\mathrm{b}.$$

\subsection{Set-valued mappings}
For a \emph{set-valued mapping} $H:\mathbb{R}^h\rightrightarrows \mathbb{R}^r$,   $\dom H$ denotes its \emph{domain}, and $\gr(H)$ denotes its \emph{graph}. The \emph{inverse} of $H$, denoted by $H\inv$, is the set-valued mapping defined by
$$(x,y)\in \gr(H\inv)\iff (y,x)\in \gr(H).$$
If $\dom H = \mathbb{R}^h$, then $H$ is called \emph{strict}. We say that $H$ is \emph{convex} if its graph is convex, \emph{closed} if its graph is closed,
and a \emph{process} if its graph is a cone.

Let $S$ be a subset of $\mathbb{R}^h$. We say that $S$ is
\begin{itemize}
 \item \emph{weakly}-${H}$-\emph{invariant} if $H(x)\cap S \neq \emptyset$ for all $x\in S$ (i.e. $S\subseteq H^{-1}(S)$).
 \item \emph{strongly}-${H}$-\emph{invariant} if $H(x)\subseteq S$ for all $x\in S$ (i.e. $H(S)\subseteq S$).
\end{itemize}

We say that a real number $\lambda$ is an \emph{eigenvalue} of the set-valued mapping $H$ if there exists a non-zero vector $x\in \mathbb{R}^h$ such that
$$\lambda x \in H(x).$$
Such a vector $x$ is then called an \emph{eigenvector} of $H$.

Duals of set-valued mappings will play an important role in our study of reachability. For a set-valued mapping $H:\mathbb{R}^r\rightrightarrows \mathbb{R}^r$, we will employ two different dual set-valued mappings $H^\circ$ and $H^-$ defined by
\begin{align*}
\gr(H^\circ)
  &= \begin{bmatrix} 0 & I_r\\ -I_r & 0\end{bmatrix} \gr(H)^\circ\\
\gr(H^-)
  &= \begin{bmatrix} 0 & I_r\\ -I_r & 0\end{bmatrix} \gr(H)^- 
\end{align*}
where $I_r$ denotes the $r\times r$ identity matrix. Note that $H^\circ$ and $H^-$ coincide in case $H$ is a process.

\subsection{Difference inclusions}
Assume that $H$ is convex and $0\in H(0)$. Consider the \emph{difference inclusion}
\begin{equation}\label{4diffinc}
x_{k+1}\in H(x_k).
\end{equation}
An infinite sequence $\{x_k\}_{k\geq 0}\subset \mathbb{R}^h$ satisfying \eqref{4diffinc} is called a \emph{solution} of \eqref{4diffinc}. The set of all initial states, from which a solution of \eqref{4diffinc} starts, will be denoted by $X(H)$. For $\ell \geq 1$, we define the sets 
$$
X_\ell(H) = H^{-\ell}\mathbb{R}^h,\qquad R_\ell(H)=H^\ell(0).$$
We obviously have
$$
X_{\ell+1}(H)\supseteq X_{\ell}(H)\text{ and }R_\ell(H)\subseteq R_{\ell+1}(H)
$$
for all $\ell\geq 1$.
The set of all states reachable from origin in finite steps will be denoted by $R(H)$. That is
$$
R(H)=\bigcup_{\ell\geq 1}R_\ell(H).
$$
Note that
$$
 X(H) \subseteq \bigcap_{\ell\geq 1}X_\ell(H).
$$
However, the equality does not hold in general. A particularly important case is when
$$
X(H)=X_\ell(H)
$$
for some $\ell\geq 1$. In this case, we say that $X(H)$ is \emph{finitely determined}. The set $X(H)$ is the largest weakly-$H$-invariant set and $R(H)$ is the smallest strongly-$H$-invariant set containing the origin.  

The difference inclusion \eqref{4diffinc} is said to be \emph{reachable} if 
$$X(H)\subseteq R(H)$$
and \emph{weakly asymptotically stable} if for each $x\in \dom H$  there exists a solution $\{x_k\}_{k\geq 0}$ of \eqref{4diffinc} with $x_0=x$ satisfying
$$\lim_{k\to \infty} x_k = 0.$$

With a slight abuse of terminology, we sometimes say that a set-valued mapping 
is reachable or weakly asymptotically stable meaning that the corresponding difference inclusion enjoys the mentioned property.

In the rest of this subsection we would like to give a summary of the crucial weak asymptotic stability results from \cite{smirnov:02}.

Suppose that $H$ in \eqref{4diffinc} is a strict convex process. The process $H$ is strict if and only if $H^+(0)=\{0\}$. Moreover, the restriction of the mapping $H^+$ to $\mathcal{W}=\dom (H^+)\cap [-\dom (H^+)]$ is a linear transformation. We will denote the largest subspace invariant under $H^+|_{\mathcal{W}}$ and contained in $\mathcal{W}$ by $\mathcal{J}$. 

\begin{theorem}\label{t:4thmsmir}
Let $H: \mathbb{R}^r\rightrightarrows \mathbb{R}^r$ be a strict convex process. Suppose that all eigenvalues of $H^+$ are less than $1$ and  all eigenvalues of $H^+|_\mathcal{J}$ are in the open unit circle. Then, $H$ is weakly asymptotically stable.
\end{theorem}
 
Theorem~\ref{t:4thmsmir} is the discrete-time analogue of a part of \cite[Thm. 8.10]{smirnov:02}. It is stated as a problem in \cite[Sect. 8.6]{smirnov:02}. In the next two remarks we will sketch a proof of it.

\begin{remark}
First, we note that in \cite{smirnov:02}, the closedness of $H$ is assumed. As long as $H$ is strict, this assumption is, in fact, redundant. It immediately follows from the definitions that $H$ is weakly asymptotically stable if and only if 
$$\bigcap_{\mu\in (0, +\infty)} \left( \bigcup_{\ell\geq 1} H^{-\ell}(\mu \mathbb{B})\right) = \mathbb{R}^n.$$
Since $H$ is a process, one can take $\mu$ out of the big parenthesis. Then, left hand side of the above equation is nothing but the recession cone of $ \bigcup_{\ell\geq 1} H^{-\ell}( \mathbb{B})$. That is, $H$ is weakly asymptotically stable if and only if 
$$0^+ \left( \bigcup_{\ell\geq 1} H^{-\ell}( \mathbb{B})\right) = \mathbb{R}^n.$$ 
This is equivalent to
$$ \bigcup_{\ell\geq 1} H^{-\ell}( \mathbb{B}) = \mathbb{R}^n.$$
Since $H$ is strict, it is possible to show that $[ H^{-\ell}( \mathbb{B})]^\circ = (H^+)^\ell(\mathbb{B})$. Then, it follows from \cite[Cor. 16.5.2]{Roc:70} that $H$ is weakly asymptotically stable if and only if
$$\bigcap_{\ell\geq 1} (H^+)^\ell(\mathbb{B}) = \{0\}.$$
If we let $\overline{H}$ denote the strict closed convex process defined as $\gr(\overline{H}) = \clo(\gr(H))$. Since $H^+ = (\overline{H})^+$, we conclude that $H$ is weakly asymptotically stable if and only if so is $\overline{H}$. 
\end{remark}

\begin{remark}
The main ingredient of the continuous-time version is \cite[Thm. 8.9]{smirnov:02} which is a consequence of \cite[Thm. 2.14 and Lemma 8.3]{smirnov:02}. 

Lemma 8.3 of \cite{smirnov:02} has a discrete-time analogue. It is possible to prove that if $\lambda \in [0,1)$ and $x\in (H-\lambda I_r)^{-k}(0)$ for some $k\geq 1$, then the sequence
 $$\left\{ \sum_{i = \max \{0, j-k\}}^j \binom{j}{i} \lambda^i x_{j-i}\right\}_{j\geq 0}$$
is a solution of \eqref{4diffinc} starting from $x$ and converging to zero.
 
 Regarding Theorem 2.14 of \cite{smirnov:02}, a discrete-time version is possible to formulate too. In Smirnov's notation, $\lambda_0(H^+)$ denotes the largest eigenvalue of $H^+$. When $\lambda > \lambda_0(H^+)$, the process $H-\lambda I_r$ is onto. Hence, $\bigcup_{\ell\geq 1} (H-\lambda I_r)^{-\ell}(0) = \mathbb{R}^r$ if and only if the strict convex process $(H-\lambda I_r)^{-1}$ is reachable. However, this is characterized in \cite{KaCa:15a}. When $\calJ = \{0\}$, the process $(H-\lambda I_r)^{-1}$ is reachable if and only if $((H-\lambda I_r)^{-1})^+$ has no eigenvectors corresponding to a non-negative eigenvalue. But this is guaranteed by the condition $\lambda > \lambda_0(H^+)$. Therefore, we proved: If $\calJ = \{0\}$ and $\lambda > \lambda_0(H^+)$, then $\bigcup_{\ell\geq 1} (H-\lambda I_r)^{-\ell}(0) = \mathbb{R}^r$.
 
 The discrete-time version of \cite[Thm. 8.9]{smirnov:02} reads
 \begin{quote}
  If $\calJ = \{0\}$ and the eigenvalues of $H^+$ are less than $1$, then $H$ is weakly asymptotically stable.
 \end{quote}
 Now its proof is immediate from the arguments of the previous two paragraphs, due to the existence of $\lambda$ with $\lambda_0(H^+) < \lambda < 1$.
 
 Once we proved the discrete-time version of \cite[Thm. 8.9]{smirnov:02}, the proof of Theorem \ref{t:4thmsmir} follows easily, since the arguments in the last four paragraph of the proof of \cite[Thm. 8.10]{smirnov:02} work in discrete-time mutatis mutandis.
\end{remark}

\subsection{Geometric control theory}
Next, we recall the crucial concepts from geometric control theory of linear systems. For details, we refer to \cite{Trent:01}.

For the linear system $\Sigma=\Sigma(A,B,C,D)$ of the form \eqref{e:4dts-1}-\eqref{e:4dts-2}, the \emph{weakly unobservable subspace} will be denoted by $\mathcal{V}^*(\Sigma)$ and the \emph{strongly reachable subspace} by $\mathcal{T}^*(\Sigma)$. The intersection  $\mathcal{V}^*(\Sigma)\cap\mathcal{T}^*(\Sigma)$ is called the \emph{controllable weakly unobservable subspace} and will be denoted by $\mathcal{R}^*(\Sigma)$.

As it was already discussed in \cite{KaCa:15a}, the subspaces
$$\mathcal{K}(\Sigma)= \im D + C\mathcal{T}^*(\Sigma) \quad\text{and}\quad \mathcal{L}(\Sigma)= \ker D \cap B^{-1}\mathcal{V}^*(\Sigma)$$
can be used to characterize the right-(left-)invertibility of the linear system $\Sigma$. Namely, the linear system $\Sigma$ is right-invertible if and only if $\mathcal{K}(\Sigma)= \mathbb{R}^s$. Similarly, it is left-invertible if and only if $\mathcal{L}(\Sigma)= \{0\}$. It is also worth noting that
$$\mathcal{K}(\Sigma)^\perp = \mathcal{L}(\Sigma^T)$$
where $\Sigma^T$ denote the dual of $\Sigma$, that is 
$$\Sigma^T=\Sigma(A^T,C^T,B^T,D^T).$$ 

For a subspace $\mathcal{U}\subseteq \mathbb{R}^m$ of dimension $k$ and an injective linear transformation $E:\mathbb{R}^h\to \mathbb{R}^n$ with $\im E=\mathcal{U}$, we denote the discrete-time linear system $\Sigma(A,BE,C,DE)$ by $(\mathcal{U}, \Sigma )$ and its weakly unobservable subspace by $\mathcal{V}^*(\mathcal{U},\Sigma)$. An important subspace of $\mathcal{V}^*(\mathcal{U},\Sigma)$ is $\mathcal{V}_g^*(\mathcal{U},\Sigma)$. It denotes the set of all initial states $x\in \mathcal{V}^*(\mathcal{U}, \Sigma )$ for which there exists a bounded sequence $\{x_i\}_{i\geq 0}$ with $x_0=x$, and a sequence $\{u_i\}_{i\geq 0}\subset \mathcal{U}$ such that
\begin{align*}
 x_{k+1} &=Ax_k+Bu_k,\\
 0 &= Cx_k+Du_k.
\end{align*}
for all $k\geq 0$.

Having mentioned all the preliminary concepts, we are ready to present our main results.

\section{Main results}
Throughout this paper, the following blanket assumption will be in force. 
\begin{asmp}\label{4asmp}
 The convex set $\mathcal{Y}\cap \im\begin{bmatrix} C &  D\end{bmatrix}$ is solid.
\end{asmp}
It can be shown exactly like the conic constraint case \cite{KaCa:15a} that this assumption does not cause any loss of generality. An immediate consequence of Assumption \ref{4asmp} is that $\mathcal{Y}$ is solid and $\begin{bmatrix} C &  D\end{bmatrix}$ is surjective.

Our aim is to carry the reachability problem to the setting of set-valued mappings and derive a characterization by combining the techniques of geometric control theory with that of convex analysis. Therefore, as a first step, we reformulate the discrete-time constrained linear system \eqref{e:4dts} as a difference inclusion
$$
 x_{k+1}\in F(x_k)
\nonumber
$$
where $F:\mathbb{R}^n\rightrightarrows \mathbb{R}^n$ is a convex set-valued mapping given by
$$
 F(x)=Ax+BD^{-1}(\mathcal{Y}-Cx),
$$
or equivalently, described by
\begin{equation}\label{e:4DefnGrF}
 \gr(F)=
\begin{bmatrix}
 I_n & 0\\ A & B
\end{bmatrix} 
\begin{bmatrix}
 C & D
\end{bmatrix}^{-1} \mathcal{Y}.
\end{equation}
%
%

Clearly, $(\Sigma, \mathcal{Y})$ is reachable if and only if $F$ is reachable, that is
$$
X(F) \subseteq R(F). 
$$
In the analysis of reachability, we will employ the following two convex processes associated to $F$:
\begin{subequations}
\begin{align}
\gr(F_{\mathrm{con}})
&=
\begin{bmatrix}
 I_n & 0\\ A & B
\end{bmatrix} 
\begin{bmatrix}
 C & D
\end{bmatrix}^{-1} \cone(\mathcal{Y})\label{e:Fcon}\\
\gr(F_{\mathrm{rec}})
&=
\begin{bmatrix}
 I_n & 0\\ A & B
\end{bmatrix} 
\begin{bmatrix}
 C & D
\end{bmatrix}^{-1} (0^+\mathcal{Y})\label{e:Frec}
\end{align}
\end{subequations}
Note that these convex processes can be obtained by using the cones $\cone(\calY)$ and $0^+\calY$ instead of $\calY$ in \eqref{e:4DefnGrF} and hence correspond to the constrained linear systems $\big(\Sigma,\cone(\mathcal{Y})\big)$ and $(\Sigma,0^+\mathcal{Y})$, respectively.

Not only $F$, $F_{\mathrm{con}}$, and $F_{\mathrm{rec}}$ but also their duals will play an important role in our study of reachability. It follows from \eqref{e:4DefnGrF}, Assumption \ref{4asmp}, and \cite[Cor. $16.3.2$]{Roc:70} that
\begin{equation}\label{e:4Fpolar}
\gr(F^\circ) = \begin{bmatrix}
  A^T & -I_n\\ B^T & 0
  \end{bmatrix}^{-1}
  \begin{bmatrix}
  C^T \\ D^T
  \end{bmatrix}\mathcal{Y}^\circ.
\end{equation}
In a similar fashion, it follows from \eqref{e:Fcon}, Assumption \ref{4asmp}, and \cite[Cor. $16.3.2$]{Roc:70} that

$$
 \gr((F_{\mathrm{con}})^\circ) = \gr(F^-) =
\begin{bmatrix}
  A^T & -I_n\\ B^T & 0
  \end{bmatrix}^{-1}
  \begin{bmatrix}
  C^T \\ D^T
  \end{bmatrix}
  \calY^-.
$$
If $\calY^\mathrm{b}$ is closed, then $\calY^\mathrm{b} = (0^+\calY)^-$. In such a case, we will denote $(F_{\mathrm{rec}})^\circ$ by $F^\mathrm{b}$. It follows from \eqref{e:Frec}, Assumption \ref{4asmp}, and \cite[Cor. $16.3.2$]{Roc:70} that

$$
\gr((F_{\mathrm{rec}})^\circ) = \gr(F^\mathrm{b})
  = \begin{bmatrix}
  A^T & -I_n\\ B^T & 0
  \end{bmatrix}^{-1}
  \begin{bmatrix}
  C^T \\ D^T
  \end{bmatrix}\mathcal{Y}^\mathrm{b}.
$$
whenever $\calY^\mathrm{b}$ is closed.

The dual set-valued mappings $F^\circ$, $F^-$, and $F^{\mathrm{b}}$ are closely related to the dual linear system $\Sigma^T$ and they can be can equivalently be described by
\bse\label{e:FpolSigmaT}
\begin{align}
F^\circ(q)&=\{A^Tq+C^Tv\mid v\in -\mathcal{Y}^\circ\text{ and }B^Tq+D^Tv=0\}\label{e:FpolSigmaT.1}\\
F^-(q)&=\{A^Tq+C^Tv\mid v\in \mathcal{Y}^+\text{ and }B^Tq+D^Tv=0\}\label{e:FpolSigmaT.2}\\
F^\mathrm{b}(q)&=\{A^Tq+C^Tv\mid v\in -\mathcal{Y}^b\text{ and }B^Tq+D^Tv=0\}.\label{e:FpolSigmaT.3}
\end{align}
\ese
Note that $F^\circ$ is a closed convex set-valued mapping but not necessarily a process, whereas both $F^-$ and $F^{\mathrm{b}}$ are closed convex processes. 

Depending on the intersection of the subspace $\mathcal{K}(\Sigma)$ and the convex set $\mathcal{Y}$, we distinguish three cases:
\begin{enumerate}
 \item \label{i:4case1}$\mathcal{K}(\Sigma)\cap \inte(\mathcal{Y})\neq \emptyset$.
 \item \label{i:4case2}$\mathcal{K}(\Sigma)\cap\mathcal{Y}=\{0\}$.
 \item \label{i:4case3}$\{0\}\subsetneqq \mathcal{K}(\Sigma)\cap\mathcal{Y}\subseteq \mathcal{K}(\Sigma)\setminus \inte(\mathcal{Y})$.
\end{enumerate}
Note that since $\mathcal{Y}$ is solid, this is an exhaustive list of possibilities. We will treat the first two cases in what follows. However, the last case will not be studied in this paper. We refer \cite{KaCa:15a} for a discussion of this case when the constraint set $\calY$ is a convex cone.

\subsection{Case \ref{i:4case1}: $\mathcal{K}(\Sigma)\cap \inte(\mathcal{Y})\neq \emptyset$}
We begin with establishing a duality relation between $R(F)$ and $X(F^\circ)$.
\begin{theorem}\label{t:4sumfullauxthm2}
Suppose that $\mathcal{K}(\Sigma)\cap \inte(\mathcal{Y})\neq \emptyset$. Then, $$X(F^\circ)=R(F)^\circ.$$ 
\end{theorem}

This theorem is a key step in carrying the reachability problem to the dual world. As a first step towards a spectral characterization of reachability in this case, we will first derive some sufficient conditions and later on discuss when these sufficient conditions are also necessary. A natural sufficient condition for the reachability of $(\Sigma,\mathcal{Y})$ is
\begin{equation}\label{e:4rfrn}
R(F)=\mathbb{R}^n, 
\end{equation}
which admits the following characterization by Theorem~\ref{t:4sumfullauxthm2}.

\begin{cor}\label{t:4sumfullthm}
Suppose that $\mathcal{K}(\Sigma)\cap \inte(\mathcal{Y})\neq \emptyset$. Then $R(F)=\mathbb{R}^n$ if and only if $X(F^\circ)=\{0\}$.
\end{cor}

This corollary provides a sufficient condition for reachability in terms of the dual set-valued mapping. However, it is far from being useful for practical purposes. Because, verifying the condition $X(F^\circ)=\pset{0}$ is a hard task as $F^\circ$ is merely a closed convex set-valued mapping and not necessarily a process in general. Yet, it is possible to provide easily verifiable spectral conditions in terms of the reachability of $F_{\mathrm{con}}$ and the stability of $F_{\mathrm{rec}}^{-1}$ under the assumption that the constraint set $\calY$ is hyperbolic.

\begin{theorem}\label{t:4sumfullthm2}
Suppose that $\mathcal{K}(\Sigma)\cap \inte(\mathcal{Y})\neq \emptyset$ and $\mathcal{Y}$ is hyperbolic. Then the following conditions are equivalent:
\begin{enumerate}
 \item \label{i:4sft21} $X(F^\circ)=\{0\}$.
 \item \label{i:4sft22} Both of the following conditions hold:
       \begin{enumerate}
        \item \label{i:4sft221} $F_{\mathrm{con}}$ is reachable.
        \item \label{i:4sft222} $(F_{\mathrm{rec}})^{-1}$ is strict and weakly asymptotically stable.
       \end{enumerate}
 \item \label{i:4sft23} All of the following conditions hold:
        \begin{enumerate}
        \item\label{4lschar} $\Sigma$ is controllable.        
        \item\label{4evchar} For all $\lambda\in [0,+\infty)$, $u\in \mathcal{Y}^+$ and $q\in \mathbb{R}^n$, 
        $$\begin{bmatrix}
                  A^T-\lambda I_n\\ B^T
                \end{bmatrix}q+\begin{bmatrix}
                 C^T\\ D^T
                \end{bmatrix}u= 0\Rightarrow q=0.$$
        \item\label{i:4sft234}$\mathcal{V}_g^*([0^+\mathcal{Y}]^\perp, \Sigma^T)=\{0\}.$
        \item\label{i:4sft233} For all $\lambda\in [0,1]$, $u\in -\mathcal{Y}^\mathrm{b}$ and $q\in \mathbb{R}^n$, 
        $$\begin{bmatrix}
                  A^T-\lambda I_n\\ B^T
                \end{bmatrix}q+\begin{bmatrix}
                 C^T\\ D^T
                \end{bmatrix}u= 0\Rightarrow q=0.$$
        \end{enumerate}
 \end{enumerate}
\end{theorem}

\begin{remark}
The hyperbolicity condition on $\calY$ is needed in order to provide spectral conditions given in the last part of the theorem. Indeed, even the input constraint reachability problem does not admit a spectral characterization (see e.g. \cite{Ng:85} for continuous-time and \cite{Ng:81} for discrete-time systems) without the hyperbolicity condition. Note that $\calY$ is naturally hyperbolic when it is a cone, a bounded set, or a polyhedral set.
\end{remark}

\begin{remark}
An important special case arises when $\mathcal{Y}$ is a cone. In this case, $\mathcal{Y}$ is readily hyperbolic. Furthermore, the cones $\mathcal{Y}^-$ and $\mathcal{Y}^\mathrm{b}$ coincide. As such, Theorem~\ref{t:4sumfullthm2} boils down to the sufficiency part of Theorem~6.3 in \cite{KaCa:15a}. 
\end{remark}

\begin{remark}
Theorem~\ref{t:4sumfullthm2} reveals why the convex output constraint case cannot be solved without studying the convex conic output constraint case first. Indeed, the condition \eqref{i:4sft221} requires investigating reachability under conic constraints even if the constraint set $\mathcal{Y}$ is not a cone.
\end{remark}

\begin{remark}
In case of conic constraints, the assumption $\mathcal{K}(\Sigma)\cap \inte(\mathcal{Y})\neq \emptyset$ yields $\mathcal{T}^*(\Sigma) + \dom (F) = \mathbb{R}^n$. In the absence of conicity, the same equality is too much to hope for. However, one can still show that $0\in \inte(\mathcal{T}^*(\Sigma) + \dom (F))$. In \cite{KaCa:15a}, we showed that for a convex process $H: \mathbb{R}^n \rightrightarrows \mathbb{R}^n$, if there exists a subspace $\mathcal{W}\subseteq F^\ell(0)$ for some $\ell\geq 0$, satisfying $\mathcal{W} \subseteq H(\mathcal{W}) $ and $\mathcal{W} + \dom (H) = \mathbb{R}^n$, we could then make the following strong conclusions
\begin{itemize}
\item Reachability of $H$ implies $R(H) = \mathbb{R}^n$,
\item $X(H)$ is finitely determined.
\end{itemize}
When we remove the conicity assumption, and let $H$ be a convex set-valued mapping only, and replace the assumption
$$\mathcal{W} + \dom (H) = \mathbb{R}^n$$
with the appropriate analogue of it
$$0\in \inte(\mathcal{W} + \dom (H)),$$
the two conclusions above both fail to hold. A simple example to this fact is given by the convex set-valued mapping $H: \mathbb{R}^n \rightrightarrows \mathbb{R}^n$ defined as
$$\gr(H) := \{(x,y)| \quad |x|\leq 1\text{ and } y \geq 2|x|\}.$$
It is easy to see that $X(H)=\{0\}$, hence $H$ is reachable. However, neither $X(F)$ is finitely determined, nor $R(F)$ is equal to $\mathbb{R}^n$.
\end{remark}

So far we provided conditions that are equivalent to the relation \eqref{e:4rfrn} and hence are sufficient for the reachability of $(\Sigma, \mathcal{Y})$. When $\calY$ is a cone, the condition $\mathcal{K}(\Sigma)\cap \inte(\mathcal{Y})\neq \emptyset$ is satisfied if and only if we have $\mathcal{K}(\Sigma)+ \mathcal{Y}=\R^s$. In this case, we know from Theorem~6.3 of \cite{KaCa:15a} that the equality \eqref{e:4rfrn} is also a necessary condition for reachability. However, its necessity does not come naturally when $\calY$ is merely a convex set under the condition $\mathcal{K}(\Sigma)\cap \inte(\mathcal{Y})\neq \emptyset$. We will discuss this in the next example.

\begin{example}
Suppose that $\Sigma$ is given by the matrices
\begin{alignat*}{2}
A = \bbm 0 & 1\\ 1 & 0\ebm \qquad&& B = \bbm 1\\ 0\ebm \\
C = \bbm 0 & 0\\ 0 & 1\ebm \qquad&& D = \bbm 1\\ 0 \ebm 
\end{alignat*}
and the constraint set is
$$\calY = [-1,1]\times [-1,1].$$
Since $\calY$ is bounded, it is hyperbolic.

It can easily be shown that $\calT^*(\Sigma) = \{0\}$, hence
$$\calK(\Sigma) = \im D = \im \bbm 1 \\ 0\ebm.$$
So $\calK(\Sigma) \cap \inte(\calY)\neq \emptyset$, however $\calK(\Sigma) + \calY\neq \mathbb{R}^2$. Moreover, we have
$$X(F) = [-1,1]\times [-1,1] \quad \text{and} \quad R(F) = [-2,2]\times [-2,2]$$
hence $(\Sigma,\calY)$ is reachable. However, Theorem~\ref{t:4sumfullthm2} fails to characterize this situation, as condition \eqref{i:4sft233} does not hold. Indeed, we have $\calY^\mathrm{b} = \mathbb{R}^2$ and for $\lambda =1$, we get
$$0 \neq \bbm 1\\ 1\ebm \in \bbm A^T- I_2 \\ B^T\ebm\inv \im \bbm C^T \\ D^T\ebm.$$
\end{example}

What is really desirable is to find some extra conditions which will enforce the equivalence of the reachability of $F$ and the condition $R(F)= \mathbb{R}^n$. For an arbitrary convex set-valued mapping $H:\mathbb{R}^n\rightrightarrows \mathbb{R}^n$, if $H$ is onto, reachable and if $X(H)$ is finitely determined, we obviously have $R(H) = \mathbb{R}^n$. The interesting fact is that for the set-valued mapping $F$, the converse also holds under mild assumptions.

\begin{prop}
Suppose that $\mathcal{K}(\Sigma)\cap \inte(\mathcal{Y})\neq \emptyset$ and $\mathcal{Y}$ is closed. Then, $X(F)$ is finitely determined. Moreover, $R(F) = \mathbb{R}^n$ if and only if $F$ is onto and reachable.  
\end{prop}

It turns out that we need to impose the stronger condition $\mathcal{K}(\Sigma)+ \mathcal{Y}=\R^s$ in order to obtain necessity of \eqref{e:4rfrn}. To elaborate on this point, we need the following auxiliary result.

\begin{lemma}\label{4sumfullauxthm}
Suppose that $\mathcal{K}(\Sigma)+\mathcal{Y}=\mathbb{R}^s$. Then, the following statements hold:
\begin{enumerate}
 \item \label{i:4sfat1}$X(F) = X_n(F)$. In particular, $X(F)$ is finitely determined.
 \item \label{i:4sfat2}$-R_\ell(F^\circ) = [X_\ell(F)]^\circ$ for all $l\geq 1$, and $R(F^\circ)=R_n(F^\circ)$. Hence, $-R(F^\circ)=X(F)^\circ$.
 \item \label{i:4sfat3}$\conv[X(F) \cup \mathcal{T}^*(\Sigma)] = \mathbb{R}^n$. Equivalently, $$R(F^\circ) \cap \mathcal{V}^*(\Sigma^T)=\{0\}.$$
\end{enumerate}
\end{lemma}

The biggest merit of Lemma~\ref{4sumfullauxthm} is that \eqref{e:4rfrn} is now a necessary condition for reachability. By Lemma~\ref{4sumfullauxthm}.\ref{i:4sfat3}, $X(F)\subseteq R(F)$ implies $\mathbb{R}^n = \conv[X(F) \cup \mathcal{T}^*(\Sigma)] \subseteq R(F)$. Therefore, we get the following result.

\begin{theorem}
 Suppose that $\mathcal{K}(\Sigma)+\mathcal{Y}=\mathbb{R}^s$. Then the following conditions are equivalent
 \begin{enumerate}[i.]
  \item $(\Sigma,\mathcal{Y})$ is reachable.
  \item $R(F)= \mathbb{R}^n$.
  \item $X(F^\circ) = \{0\}$.
 \end{enumerate}
\end{theorem}

Combining the above theorem with Theorem \ref{t:4sumfullthm2}, we obtain a spectral characterization for reachability of $(\Sigma,\calY)$.

The next section is devoted to the case where the subspace $\mathcal{K}(\Sigma)$ and the output constraint set have a trivial intersection.

\subsection{Case \ref{i:4case2}: $\mathcal{K}(\Sigma)\cap\mathcal{Y}=\{0\}$}

Similar to the conic constraint set, the reachable set of states $R(F)$ admits a simple characterization for this case.

\begin{theorem}\label{t:RFTstar}
Suppose that $\mathcal{K}(\Sigma)\cap\mathcal{Y}=\{0\}$.
Then $R(F) = \mathcal{T}^*(\Sigma)$.
\end{theorem}

Based on this characterization, we state necessary and sufficient conditions for reachability in the following theorem.

\begin{theorem}\label{t:Zint}
Suppose that $\mathcal{K}(\Sigma)\cap\mathcal{Y}=\{0\}$. Then the following statements are equivalent:
\begin{enumerate}[i.]
 \item \label{4zintt:1}The system $(\Sigma,\mathcal{Y})$ is reachable.
 \item \label{4zintt:2}$X(F)=\mathcal{V}^*(\Sigma)=\mathcal{R}^*(\Sigma)$.
\end{enumerate}
\end{theorem}



\section{Proofs}
Before we present the proofs, we would like to prove a few auxiliary results, which will turn out to be very useful in the proofs of the main theorems.

Let $\Lambda_1= B$, $\Gamma_1 = C$ and $\Theta_1= D$, and for all $\ell \geq 2$ define the following matrices recursively
\begin{equation}\label{4galathe}
 \Gamma_{\ell+1} = \begin{bmatrix}
                    \Gamma_\ell A \\ C 
                   \end{bmatrix},
 \quad
 \Lambda_{\ell+1} = \begin{bmatrix}
                     A \Lambda_\ell & B
                    \end{bmatrix},
 \quad
 \Theta_{\ell+1} = \begin{bmatrix}
                    C\Lambda_\ell & D\\
                    \Theta_\ell & 0
                   \end{bmatrix}.
\end{equation}
Note that using these matrices, for the strongly reachable subspace of $\Sigma$ and weakly unobservable subspace of $\Sigma^T$, one could write
$$
 \mathcal{T}^*(\Sigma) = \Lambda_n\ker \Theta_n\text{ and }  \mathcal{V}^*(\Sigma^T) = (\Lambda_n^T)^{-1}\im \Theta_n^T.
$$
The next lemma collects technical auxiliary results that will be employed in the proofs of the main results. 
\begin{lemma}\label{l:4SLem}
Suppose that $\mathcal{K}(\Sigma)\cap \inte(\mathcal{Y})\neq \emptyset$. Then, the following statements hold:
\begin{enumerate}
 \item \label{i:4SLem3} $(F^\circ)^{-1}\calV^*(\Sigma^T) \subseteq \calV^*(\Sigma^T)$.
 \item \label{i:4SLem1} $[\im \Theta_\ell + \Gamma_\ell \mathcal{T}^*(\Sigma)]  \cap \inte(\mathcal{Y}^\ell) \neq \emptyset,$
for all $\ell \geq 1$.
 \item \label{i:4SLem2} $F^\circ(0) \cap \calV^*(\Sigma^T)$ is compact. In particular, if $\calK(\Sigma)+\calY = \mathbb{R}^s$ then $F^\circ(0) \cap \calV^*(\Sigma^T) = \{0\}$.
 \item \label{i:4SLem4} $X_n(F^\circ) \subseteq \calV^*(\Sigma^T)$.
 \item \label{i:4sfalem1}$(F^\circ)^{-\ell}(\mathcal{V}^*(\Sigma^T))=[F^\ell(\mathcal{T}^*(\Sigma))]^\circ$ for all $\ell \geq 0$.
 \item \label{i:4sfalem2}$X(F^\circ)=\bigcap_{\ell\geq 1}(F^\circ)^{-\ell}\mathcal{V}^*(\Sigma^T)=\bigcap_{\ell\geq 1}X_\ell(F^\circ)$.
\end{enumerate}
\end{lemma}
\begin{proof} {\em \ref{i:4SLem3}:} This follows immediately from \eqref{e:FpolSigmaT.1} and definition of $\calV^*(\Sigma^T)$.\\
 
{\em \ref{i:4SLem1}:} Before proceeding to the actual proof, we note that $[\im \Theta_\ell + \Gamma_\ell \mathcal{T}^*(\Sigma)]  \cap \inte(\mathcal{Y}^\ell) \neq \emptyset$ if and only if
$$0\in \inte [\im \Theta_\ell + \Gamma_\ell \mathcal{T}^*(\Sigma) + \mathcal{Y}^\ell].$$
However, this is equivalent to
\beq\label{e:uefa-kupayi-al}
\ker \Theta_\ell^T \cap (\Gamma_\ell^T)^{-1} \mathcal{V}^*(\Sigma^T) \cap (\mathcal{Y}^-)^\ell = \{0\}.
\eeq
We will prove this latter equivalent condition by induction on $\ell$.

For $\ell=1$, this is already equivalent to the hypothesis $\mathcal{K}(\Sigma)\cap \inte(\mathcal{Y})\neq \emptyset$ and hence readily follows.

Suppose that \eqref{e:uefa-kupayi-al} holds for $\ell = k$. For $\ell = k+1$,  the set
$$ \im\Theta_{k+1} + \Gamma_{k+1}\mathcal{T}^*(\Sigma)+\mathcal{Y}^{k+1}$$
is equal to
$$\im \begin{bmatrix}
                    C\Lambda_k & D\\
                    \Theta_k & 0
                    \end{bmatrix}
                + 
             \begin{bmatrix}
                    CA^k \\ \Gamma_k
                   \end{bmatrix}\mathcal{T}^*(\Sigma)
                +    
            \mathcal{Y}\times \mathcal{Y}^k  
$$
in view of \eqref{4galathe}.
If we let $\bbm u_0\\ \bar{u}\ebm\in \mathcal{Y}^-\times (\mathcal{Y}^-)^k$ be such that 
$$
\bbm u_0\\ \bar{u}\ebm\in    \ker \begin{bmatrix}
                    \Lambda_k^TC^T & \Theta_k^T\\
                    D^T & 0
                    \end{bmatrix}
                \cap
            \begin{bmatrix}
                    (A^T)^kC^T & \Gamma_k^T
                   \end{bmatrix}^{-1}\mathcal{V}^*(\Sigma^T)
$$
then we have
$$ u_0\in \ker D^T\cap \mathcal{Y}^- \text{ and } (F^-)^k(C^Tu_0)\cap\mathcal{V}^*(\Sigma^T)\neq \emptyset.$$
Hence, it follows from the first part of the lemma that $C^Tu_0\in \mathcal{V}^*(\Sigma)$. Thus, we get $u_0\in \mathcal{L}(\Sigma^T)\cap \mathcal{Y}^-= \{0\}$. Therefore, the cone
$$\ker\Theta_{k+1}^T \cap (\Gamma_{k+1}^T)^{-1}\mathcal{V}^*(\Sigma^T)\cap(\mathcal{Y}^-)^{k+1}$$
is equal to
$$\{0\}\times(\ker \Theta_k^T\cap(\Gamma_k^T)^{-1}\mathcal{V}^*(\Sigma^T)\cap  (\mathcal{Y}^-)^k)$$
which is $\{0\}$ by the induction hypothesis. Hence, the result follows.\\
 
{\em \ref{i:4SLem2}:} We argue as follows
 \begin{align*}
  F^\circ(0)\cap \mathcal{V}^*(\Sigma^T)
    &= C^T[\ker D^T\cap (-\mathcal{Y}^\circ)]\cap \mathcal{V}^*(\Sigma^T)\\
    &= C^T[\ker D^T\!\cap\! (-\mathcal{Y}^\circ)\!\cap\! (C^T)^{-1}\mathcal{V}^*(\Sigma^T)]\\
    &= -C^T[\mathcal{L}(\Sigma^T)\cap\mathcal{Y}^\circ].
 \end{align*}
Now, $\mathcal{L}(\Sigma^T)\cap\mathcal{Y}^\circ$ is compact by the hypothesis of the lem\-ma. Then, it follows from \cite[Thm. 9.1]{Roc:70} that $C^T[\mathcal{L}(\Sigma^T)\cap\mathcal{Y}^\circ]$ is also compact since $\ker C^T\cap \ker D^T =\{0\}$. The rest follows from the fact that $\mathcal{L}(\Sigma^T)\cap\mathcal{Y}^\circ = \{0\}$ whenever $\calK(\Sigma)+\calY = \mathbb{R}^s$.\\
 
{\em \ref{i:4SLem4}:} Let $F_0^T$ denote the linear process corresponding to the constrained system $(\Sigma^T, \{0\})$. Then, we have $\gr(F^\circ) \subseteq \gr(F_0^T)$. This means that $X_n(F^\circ)\subseteq X_n(F_0^T)$. 
It follows from \cite[Thm. 7.12]{Trent:01} that $\calV^*(\Sigma^T) = X(F_0^T) = X_n(F_0^T)$. Hence, the result holds.\\

{\em \ref{i:4sfalem1}:} We first claim that
\begin{equation}\label{4claim0}
 F^\ell(\mathcal{T}^*(\Sigma))=\Lambda_{n+\ell}\Theta_{n+\ell}^{-1}(\mathcal{Y}^\ell\times \{0\}) 
\end{equation}
for all $\ell\geq 0$. We prove this statement by induction on $\ell$.

For $\ell = 0$, this follows from the definition of $\calT^*(\Sigma)$. 
Suppose that \eqref{4claim0} holds for $\ell= k$. For $\ell= k+1$, we have
\begin{align*}
 F^{k+1}(\mathcal{T}^*(\Sigma))
 &= F(F^k(\mathcal{T}^*(\Sigma)))\\
 &= \{ Ax + Bu\mid x\in \Lambda_{n+k}\Theta_{n+k}^{-1}(\mathcal{Y}^k\times \{0\})\\
 &~ \qquad \text{ and }Cx+Du\in \mathcal{Y}\}
\end{align*}
by the induction hypothesis and the definition of $F$. That is,
\begin{align*}
 \!\!\!\!\!\!\!\!\!\!F^{k+1}(\mathcal{T}^*(\Sigma))
 &= \{ A\Lambda_{n+k}\bar{u} + Bu \mid \Theta_{n+k}\bar{u}=\mathcal{Y}^k\times \{0\}\\
 &~ \qquad \text{ and }C\Lambda_{n+k} \bar{u}+Du\in \mathcal{Y}\}\\
 &= \begin{bmatrix}A\Lambda_{n+k} & B\end{bmatrix} \!\begin{bmatrix} C\Lambda_{n+k} & D\\ \Theta_{n+k} & 0\end{bmatrix}^{-1}\!\!\!(\mathcal{Y}^{k+1}\!\times\! \{0\})\\
 &= \Lambda_{n+k+1}\Theta_{n+k+1}^{-1}(\mathcal{Y}^{k+1}\times \{0\}).
\end{align*}
This proves \eqref{4claim0} by induction. For the rest, note first that one could also write $F^{\ell}(\mathcal{T}^*(\Sigma))$ as
$$
F^{\ell}(\mathcal{T}^*(\Sigma))=
\Lambda_{n+\ell} \begin{bmatrix}
                \Gamma_\ell \Lambda_n & \Theta_\ell\\
                \Theta_n & 0
              \end{bmatrix}^{-1}(\mathcal{Y}^{\ell}\times \{0\}).
$$
Then it follows from the second part of the lemma and \cite[Cor. $16.3.2$]{Roc:70} that
$$
[F^{\ell}(\mathcal{T}^*(\Sigma))]^\circ = (\Lambda_{n+\ell}^T)^{-1}\Theta_{n+\ell}^T[(\mathcal{Y}^\circ)^\ell\times \mathbb{R}^{ns}].
$$
However, the right hand side is nothing but $(F^\circ)^{-\ell}(\mathcal{V}^*(\Sigma^T)$.\\

{\em \ref{i:4sfalem2}:} The proof consists of two parts. First, we will prove that
\beq\label{e:cas-davasi-namusumuzdur}
\bigcap_{\ell\geq 1}(F^\circ)^{-\ell}\mathcal{V}^*(\Sigma^T)=\bigcap_{\ell\geq 1}X_\ell(F^\circ)
\eeq
and that
\beq\label{e:bizi-kume-dusurun}
X(F^\circ)=\bigcap_{\ell\geq 1}(F^\circ)^{-\ell}\mathcal{V}^*(\Sigma^T).
\eeq
From the definitions, we have
\beq\label{e:azizciler-vs-adnanhocacilar2}
\bigcap_{\ell\geq 1}(F^\circ)^{-\ell}\mathcal{V}^*(\Sigma^T)\subseteq \bigcap_{\ell\geq 1}X_\ell(F^\circ).
\eeq
It follows from the fourth part of the lemma that 
$$X_{\ell+n}(F^\circ) \subseteq (F^\circ)^{-\ell}\mathcal{V}^*(\Sigma^T)$$
for all $\ell \geq 1$. Then, we have
$$
\bigcap_{\ell\geq 1}X_\ell(F^\circ) = \bigcap_{\ell\geq 1}X_{\ell+n}(F^\circ)  \subseteq \bigcap_{\ell\geq 1} (F^\circ)^{-\ell}\mathcal{V}^*(\Sigma^T).
$$
Together with \eqref{e:azizciler-vs-adnanhocacilar2}, this results in \eqref{e:cas-davasi-namusumuzdur}. From \eqref{e:cas-davasi-namusumuzdur}, we have the inclusion
$$X(F^\circ) \subseteq \bigcap_{\ell\geq 1}(F^\circ)^{-\ell}\mathcal{V}^*(\Sigma^T).$$
In order to prove the reverse inclusion, we will show that the infinite intersection is weakly-$F^\circ$-invariant. This would imply \eqref{e:bizi-kume-dusurun} as $X(F^\circ)$ is the largest weakly-$F^\circ$-invariant set.
Let $q$ be an element of $\bigcap_{\ell\geq 1}(F^\circ)^{-\ell}\mathcal{V}^*(\Sigma^T)$. Our aim is to show that 
\begin{equation}\label{e:4aim}
F^\circ(q) \cap  \bigcap_{\ell\geq 1}(F^\circ)^{-\ell}\mathcal{V}^*(\Sigma^T)\! =\!  \bigcap_{\ell\geq 1}\big(F^\circ(q) \cap (F^\circ)^{-\ell}\mathcal{V}^*(\Sigma^T)\big) \!\neq\!\emptyset. 
\end{equation}
Define $Z_\ell=F^\circ(q)\cap (F^\circ)^{-\ell}\mathcal{V}^*(\Sigma^T)$. Since $q\in (F^\circ)^{-\ell}\mathcal{V}^*(\Sigma^T)$ for all $\ell\geq 1$, the set $Z_\ell$ is non-empty. It follows from the fifth part of the lemma that  $Z_{\ell+1} \subseteq Z_\ell$ for all $\ell\geq 1$. This means that \eqref{e:4aim} follows from Helly's Theorem, \cite[Cor. 21.3.2]{Roc:70} if each $Z_\ell$ is a compact convex set. Convexity of $Z_\ell$ is obvious. Next, we show that it is compact. As $\gr(F^\circ)$ is closed, $F^\circ(q)$ is also closed. By the fifth part of this lemma, the set $(F^\circ)^{-\ell}\mathcal{V}^*(\Sigma^T)$ is closed too. Hence $Z_\ell$ is closed for all $\ell\geq 1$. Due to first part of this lemma, we have $(F^\circ)^{-\ell}\mathcal{V}^*(\Sigma^T) \subseteq \mathcal{V}^*(\Sigma^T)$ for all $\ell\geq 1$. So we have $Z_\ell\subseteq \calV^*(\Sigma^T)$ and $F^\circ(q)\cap \mathcal{V}^*(\Sigma^T)\neq \emptyset$. We further claim that $F^\circ(q)\cap \mathcal{V}^*(\Sigma^T)$ is bounded. In order to see this, suppose on the contrary that $F^\circ(q)\cap \mathcal{V}^*(\Sigma^T)$ is unbounded. Then there exists an unbounded sequence $\{x_i\}_{i\geq 0}\subset F^\circ(q)\cap \mathcal{V}^*(\Sigma^T)$. So that $(q,x_i)\in \gr(F^\circ)\cap (\mathcal{V}^*(\Sigma^T)\times \mathcal{V}^*(\Sigma^T))$. Let  $\alpha\in [1, \infty)$. After replacing with a subsequence if necessary,  $\|x_i\|> \alpha$ for all $i\geq 0$ and $\{(\frac{\alpha}{\|x_i\|} q, \frac{\alpha}{\|x_i\|}x_i)\}$ converges to $ (0, \bar{x})\in \gr(F^\circ)\cap (\mathcal{V}^*(\Sigma^T)\times \mathcal{V}^*(\Sigma^T))$ with $\|\bar{x}\|=\alpha$. Since $\alpha$ can be arbitrarily large, this contradicts the third part of this lemma. Hence, $Z_\ell$ is compact for all $\ell\geq 1$.
\qquad\end{proof}

Now we are ready to prove the main results.

\subsection{Proof of Theorem \ref{t:4sumfullauxthm2}}
We first claim that
\begin{equation}\label{e:4rfts}
 \bigcup_{\ell \geq 1} R_\ell(F) = \bigcup_{\ell \geq 1} F^\ell(\mathcal{T}^*(\Sigma))
\end{equation}
The inclusion 
$$
\bigcup_{\ell \geq 1} R_\ell(F) \subseteq \bigcup_{\ell \geq 1} F^\ell(\mathcal{T}^*(\Sigma))
$$
is immediate. Therefore, it remains to prove the reverse inclusion. Note that
$\mathcal{T}^*(\Sigma)\subseteq R_n(F).$
Then, we have 
$$
F^\ell(\mathcal{T}^*(\Sigma))\subseteq R_{\ell+n}(F)
$$
for all $\ell\geq 1$. Taking union over all $\ell\geq 1$, we get
$$
\bigcup_{\ell \geq 1} F^\ell(\mathcal{T}^*(\Sigma)) \subseteq \bigcup_{\ell \geq 1} R_{n+\ell}(F).
$$
Then, \eqref{e:4rfts} follows from the fact that $\bigcup_{\ell \geq 1} R_{n+\ell}(F) = R(F)$. Dualizing \eqref{e:4rfts}, we get
$$
R(F)^\circ = \left[\bigcup_{\ell \geq 1} F^\ell(\mathcal{T}^*(\Sigma))\right]^\circ.
$$
Then, it follows from \cite[Cor. 16.5.2]{Roc:70} that
$$
R(F)^\circ = \bigcap_{\ell \geq 1} [F^\ell(\mathcal{T}^*(\Sigma)]^\circ.
$$
In view of Lemma~\ref{l:4SLem}.\ref{i:4sfalem1}, this results in
$$
R(F)^\circ = \bigcap_{\ell \geq 1} (F^\circ)^{-\ell}\mathcal{V}^*(\Sigma).
$$
Finally, we obtain
$$
R(F)^\circ = X(F^\circ).
$$
from Lemma~\ref{l:4SLem}.\ref{i:4sfalem2}.\qquad\endproof

\subsection{Proof of Theorem \ref{t:4sumfullthm2}}
{\em \ref{i:4sft21} $\Rightarrow$ \ref{i:4sft23}:} As $\calY^-\subseteq \calY^\circ$, \eqref{4lschar} and \eqref{4evchar} readily follow. To show that \eqref{i:4sft234} and \eqref{i:4sft233} hold. Let $q\in \calV^*_g([0^+\calY]^\perp, \Sigma^T)$ or
$$q\in \bbm A^T-\lambda I \\ B^T\ebm^{-1} \bbm C^T \\ D^T\ebm \calY^\mathrm{b}$$
for some $\lambda\in [0,1]$. Then, there exist $\{v_i\}_{i\geq 0}\in \mathcal{Y}^\mathrm{b}$ and a bounded $\{q_i\}_{i\geq 0}\subset \mathbb{R}^n$ with $q_0 = q$ such that
\begin{align*}
 q_{i+1} &=A^Tq_i-C^Tv_i\\
 0 &= B^Tq_i-D^Tv_i
\end{align*} 
for all $i\geq 0$. Since $\begin{bmatrix}
                  C & D
                 \end{bmatrix}^T$ is injective, $\{v_i\}_{i\geq 0}$ must be bounded too. In view of hyperbolicity of $\mathcal{Y}$, there must exist $\mu\in (0,1]$ such that $\{\mu v_i\}_{i\geq 0}\subseteq \mathcal{Y}^\circ$.
This means that $\{\mu q_i\}_{i\geq 0}$ is contained in $ X(F^\circ)$. Since $ X(F^\circ)=\pset{0}$, we get $q_i=0$ for all $i\geq 0$.\\

{\em \ref{i:4sft23} $\Rightarrow$ \ref{i:4sft22}:} Since we have $(F_{\mathrm{con}})^-  = F^-$, it follows from Theorem~6.3 in \cite{KaCa:15a} that \eqref{4lschar} and \eqref{4evchar} imply that $F_{\mathrm{con}}$ is reachable. Next, we will show that \eqref{i:4sft234} and \eqref{i:4sft233} imply strictness and weak asymptotic stability of $(F_{\mathrm{rec}})^{-1}$. 

Since $\calY$ is hyperbolic, $\calY^\mathrm{b}$ is closed. Hence, $(F_{\mathrm{rec}}^{-1})^+ = (F^\mathrm{b})^{-1}$. Moreover, $(F^\mathrm{b})^{-1}$ is closed and by \eqref{i:4sft233}, $(F^\mathrm{b})^{-1}(0) =\{0\}$. Equivalently, $F_{\mathrm{rec}}^{-1}$ is strict. 

For the weak asymptotic stability of $(F_{\mathrm{rec}})^{-1}$, we will employ Theorem~\ref{t:4thmsmir}. The condition \eqref{i:4sft233} immediately implies that all eigenvalues of $(F^\mathrm{b})^{-1}$ are less than $1$. Let 
$$\mathcal{W} = \im(F^\mathrm{b})\cap [-\im(F^\mathrm{b})].$$
Since $(F^\mathrm{b})^{-1}(0)=\{0\}$, the restriction of $(F^\mathrm{b})^{-1}$ to $\mathcal{W}$ is a linear map. Let $\mathcal{J}\subseteq \mathcal{W}$ be the largest subspace invariant under $(F^\mathrm{b})^{-1}|_\mathcal{W}$. We would like to show that all eigenvalues of the linear map $\Gamma=(F^\mathrm{b})^{-1}|_{\mathcal{J}}$ are in the open unit circle. Suppose on the contrary that $\Gamma$ has eigenvalues outside the open unit circle. Since $\mathcal{J}$ is $\Gamma$ invariant, one can then decompose $\mathcal{J}$ into two subspaces $\mathcal{J}_1$ and $\mathcal{J}_2\neq \{0\}$ such that $\mathcal{J}=\mathcal{J}_1\oplus\mathcal{J}_2$ where $\Gamma|_{\mathcal{J}_1}$ has only eigenvalues in the open unit circle and $\bar{\Gamma}=\Gamma|_{\mathcal{J}_2}$ has only eigenvalues outside the open unit circle. Note that $\bar{\Gamma}:\mathcal{J}_2\to \mathcal{J}_
2$ is an isomorphism. Hence, $\bar{\Gamma}^{-1}:\mathcal{J}_2\to \mathcal{J}_2$ is also an isomorphism with eigenvalues in the closed unit circle. Then, $q_{k+1}\in \bar{\Gamma}^{-1}(q_k)$ has at least one bounded trajectory.
However, this contradicts \eqref{i:4sft234} as we have
\begin{align*}
\gr(\bar{\Gamma}^{-1})
&\subseteq \gr(F^\mathrm{b})\cap[-\gr(F^\mathrm{b})]\\
&=\begin{bmatrix}
A^T & -I\\
B^T & 0
\end{bmatrix}^{-1}
\begin{bmatrix}
C^T\\
D^T
\end{bmatrix}[0^+\mathcal{Y}]^\perp. 
\end{align*}
Therefore, $\mathcal{J}_2=\{0\}$ and the hypothesis of Theorem~\ref{t:4thmsmir} is satisfied. Consequently, Theorem~\ref{t:4thmsmir} implies that $(F_{\mathrm{rec}})^{-1}$ is weakly asymptotically stable.\\

{\em \ref{i:4sft22} $\Rightarrow$ \ref{i:4sft21}:} If $F_{\mathrm{con}}$ is reachable, then $R(F_{\mathrm{con}}) = \mathbb{R}^n$, by Theorem~6.3 in \cite{KaCa:15a}.
In particular, we have $R_\ell(F_{\mathrm{con}}) = \mathbb{R}^n$ for some $\ell\geq 1$ since $F_{\mathrm{con}}$ is a process. Then, we get $0 \in \inte(R_\ell(F))$ for the same $\ell$. In other words, there exists a non-negative real number $\mu$ such that $\mu \mathbb{B}\subseteq R(F)$. Since $\gr(F_{\mathrm{rec}})\subseteq \gr(F)$, we have
\begin{equation}\label{e:FrecMuB}
\bigcup_{\ell\geq 1} F_{\mathrm{rec}}^\ell (\mu\mathbb{B}) \subseteq R(F). 
\end{equation}
However, the left hand side of \eqref{e:FrecMuB} is $\mathbb{R}^n$ since $(F_{\mathrm{rec}})^{-1}$ is strict and weakly asymptotically stable. Hence, $R(F) = \mathbb{R}^n$. Then, the result follows from Theorem \ref{t:4sumfullauxthm2}.\qquad\endproof

\subsection{Proof of Lemma~\ref{4sumfullauxthm}}

{\em \ref{i:4sfat1}:} From the definitions of $X_n(F)$ and $X(F)$, we already know that $X(F)\subseteq X_n(F)$ and $X(F)$ is the largest weakly-$F$-invariant convex set. Therefore, it is enough to show that $X_n(F)$ is weakly-$F$-invariant. This is equivalent to proving $X_n(F)=X_{n+1}(F)$. A priori we have $X_n(F)\supseteq X_{n+1}(F)$. To prove the reverse inclusion, let $x\in X_n(F)$. Then, there exist $x_0,x_1,\dots, x_n$ and $u_0,u_1,\dots, u_{n-1}$ such that $x_0=x$ and
\begin{align*}
x_{i+1} &= Ax_i + Bu_i\\
\mathcal{Y} &\ni Cx_i + Du_i
\end{align*}
for all $i$ with $0\leq i \leq n-1$. Since $\mathcal{K}(\Sigma)+\mathcal{Y}=\mathbb{R}^s$, there exists $\bar{x}\in \mathcal{T}^*(\Sigma)$, $u\in \mathbb{R}^m$ and $y\in \mathcal{Y}$ such that
$$-C\bar{x} - Du + y = Cx_n.$$
In other words, we have
\begin{equation}\label{4xnxbar}
C(x_n+\bar{x})+Du \in \mathcal{Y}. 
\end{equation}
Since $\bar{x}\in \mathcal{T}^*(\Sigma)$, there exist 
$$\bar{x}_0,\bar{x}_1,\dots, \bar{x}_n\text{ and }\bar{u}_0,\bar{u}_1,\dots, \bar{u}_{n-1}$$
such that $\bar{x}_0 = 0$, $\bar{x}_n = \bar{x}$, and
\begin{align*}
\bar{x}_{i+1} &= A\bar{x}_i + B\bar{u}_i\\
0 &= C\bar{x}_i + D\bar{u}_i
\end{align*}
for all $i$ with $0\leq i \leq n-1$. This means that we get $x_n+\bar{x}\in F^n(x)$ by starting from $x$ and applying the sequence of inputs $\{u_i+\bar{u}_i\}_{0\leq i\leq n-1}$. Then, the relation \eqref{4xnxbar} implies $x\in X_{n+1}(F)$. Consequently, $X_n(F) = X_{n+1}(F)$ and the result follows.\\

{\em \ref{i:4sfat2}:} Using the matrices in \eqref{4galathe}, we write
$$X_\ell (F)\!=\!\Gamma_\ell^{-1}[\im\Theta_\ell +\mathcal{Y}^\ell]\text{ and }R_\ell (F^\circ)\!=\!\Gamma_\ell^T[\ker\Theta_\ell^T \cap(-\mathcal{Y}^\circ)^\ell].$$
Then, we obtain
$$[X_\ell(F)]^\circ = -R_\ell (F^\circ)$$
from Lemma~\ref{l:4SLem}.\ref{i:4SLem1} and \cite[Cor. $16.3.2$]{Roc:70}.
The rest follows immediately from the first part of this lemma.\\

{\em \ref{i:4sfat3}:} Using the second part of this lemma, we observe that
\begin{align*}
 \big(\conv [X(F)\cup \mathcal{T}^*(\Sigma)]\big)^\circ &= X(F)^\circ\cap \mathcal{V}^*(\Sigma^T)\\
 &= -[R(F^\circ)\cap \mathcal{V}^*(\Sigma^T)].
\end{align*}
From Lemma~\ref{l:4SLem}.\ref{i:4SLem2}, we have $F^\circ(0)\cap \calV^*(\Sigma^T) = \{0\}$. Then, we must have $R(F^\circ)\cap \mathcal{V}^*(\Sigma^T)=\{0\}$ too since otherwise we would have
$(F^\circ)^N(q)\cap \mathcal{V}^*(\Sigma)\neq \emptyset$
for some non-zero $q\in F^\circ(0)$ and $N\geq 1$. This would mean that $q\in F^\circ(0)\cap\mathcal{V}^*(\Sigma^T)$ and hence lead to a contradiction. Therefore, we have
$$R(F^\circ)\cap \mathcal{V}^*(\Sigma^T)=\{0\}$$
and the result follows.\qquad\endproof

\subsection{Proof of Theorem \ref{t:RFTstar}}
The inclusion $\calT^*(\Sigma) \subseteq R(F)$ is obvious. In order to prove the reverse inclusion, let $x\in R(F)$. Then there exists $\ell\geq 1$, $\{u_k\}_{0\leq k\leq \ell-1}$ and $\{x_k\}_{0\leq k \leq \ell}$ with $x_0=0$  and $x_\ell=x$, such that
\begin{align*}
x_{k+1} &= Ax_k + Bu_k\\
\calY &\ni Cx_k + Du_k
\end{align*}
for all $k$ with $0\leq k \leq \ell-1$. The hypothesis $\calK(\Sigma)\cap \calY = \{0\}$ imposes $x_k\in \calT^*(\Sigma)$ for all $k$. Hence, $x\in \calT^*(\Sigma)$ and the equality holds.

\subsection{Proof of Theorem \ref{t:Zint}}
{\em\ref{4zintt:1} $\Rightarrow$ \ref{4zintt:2}:} If $(\Sigma,\calY)$ is reachable, then by Theorem \ref{t:RFTstar}, $X(F)\subseteq \calT^*(\Sigma)$. Since we always have $\calV^*(\Sigma)\subseteq X(F)$, we get $\calV^*(\Sigma) = \calR^*(\Sigma)$. In order to prove the equality of $X(F)$ and $\calV^*(\Sigma)$, let $x\in X(F)$. Then there exists, $\{u_k\}_{k\geq 0}$ and $\{x_k\}_{k\geq 0}$ with $x_0=x$, such that
\begin{align*}
x_{k+1} &= Ax_k + Bu_k\\
\calY \ni y_k &= Cx_k + Du_k
\end{align*}
for all $k\geq 0$. However, the hypothesis $\calK(\Sigma)\cap \calY = \{0\}$ imposes $y_k=0$ for all $k\geq 0$. Therefore, $x\in \calV^*(\Sigma)$, and so $X(F) = \calV^*(\Sigma)$.

{\em\ref{4zintt:2} $\Rightarrow$ \ref{4zintt:1}:} Obvious.


\section{Conclusion}
We gave a characterization of the reachability for discrete-time linear systems with convex output constraints. It extends all previously known char\-ac\-ter\-i\-za\-tion in the literature, as well as our results in \cite{KaCa:15a} on controllability of discrete-time linear systems with conic output constraints. Our results justify why the conic output constraint case must be handled first, before attempting a characterization of the convex output constraint case. In fact, we prove that convex set-valued mapping $F$ is reachable if and only if the convex process $F_{\mathrm{con}}$ is reachable and the convex process $(F_{\mathrm{rec}})^{-1}$ is weakly asymptotically stable. This result is important from two aspects. First, it gives a characterization of the reachability of a less structured set-valued mapping in terms of the properties of two more structured set-valued mappings, namely processes. Second, it reveals the relationship between the reachability problem and a stability problem. Hence, making the connection between different aspects of constrained systems more clear.

An interesting future work might be attempting to solve the constrained stabilization problem in this setting, i.e. by combining techniques of geometric control theory and with that of convex analysis. In particular, it would be interesting to give a characterization of null-controllability for such systems.

\bibliographystyle{siam}        
\bibliography{conic-revision} 

\end{document}